\documentclass[a4paper,11pt,]{article}
\usepackage{}
\usepackage{hyperref}
\usepackage{mathrsfs}
\usepackage{amsthm}
\usepackage{mathrsfs}
\usepackage{amsfonts}
\usepackage{amsmath}
\usepackage{amssymb}
\usepackage{amsthm}
\usepackage{enumerate}
\usepackage[mathscr]{eucal}
\usepackage{eqlist}
\usepackage{graphicx}

\usepackage[body={16.4true cm, 23.5true cm}]{geometry}

\newtheorem{Theorem}{Theorem}[section]

\newtheorem{Lemma}[Theorem]{Lemma}

\newtheorem{Remark}[Theorem]{Remark}

\numberwithin{equation}{section} \allowdisplaybreaks
\usepackage{mathdots}

\renewcommand\abstract{{\bf Abstract}}

\begin{document}
\title{Uniqueness of critical points of solutions to the mean curvature equation with Neumann and Robin boundary conditions\footnote{\footnotesize The work is supported by National Natural Science Foundation of China (No.11401307, No.11401310), High level talent research fund of Nanjing Forestry University (G2014022) and Postgraduate Research \& Practice Innovation Program of Jiangsu Province (KYCX17\_0321). The second author is sponsored by Qing Lan Project of Jiangsu Province.}}

\author{Haiyun Deng$^{1}$\footnote{\footnotesize Corresponding author E-mail: haiyundengmath1989@163.com, Tel.: +86 15877935256.}, Hairong Liu$^{2}$ \& Long Tian$^{1}$\\[12pt]
\small \emph {$^{1}$School of Science, Nanjing University of Science and Technology, Nanjing, 210094, China;}\\
\small \emph {$^{2}$School of Science, Nanjing Forestry University, Nanjing, 210037, China}\\}
\date{}
\maketitle

\renewcommand{\labelenumi}{[\arabic{enumi}]}

\begin{abstract}{\bf:}
 In this paper, we investigate the critical points of solutions to the prescribed constant mean curvature equation with Neumann and Robin boundary conditions respectively in a bounded smooth convex domain $\Omega$ of $\mathbb{R}^{n}(n\geq2)$. Firstly, we show the non-degeneracy and uniqueness of the critical points of solutions in a planar domain by using the local Chen \& Huang's comparison technique and the geometric properties of approximate surfaces at the non-degenerate critical points. Secondly, we deduce the uniqueness and non-degeneracy of the critical points of solutions in a rotationally symmetric domain of $\mathbb{R}^{n}(n\geq3)$ by the projection of higher dimensional space onto two dimensional plane.
\end{abstract}



\section{Introduction and main results}
~~~~In this paper we consider the following mean curvature equation
\begin{equation}\label{1.1}\begin{array}{l}
\mbox {div}(\frac{\nabla u}{\sqrt{1+|\nabla u|^2}})=H~~\mbox{in}~~\Omega,
\end{array}\end{equation}
with Neumann and Robin boundary conditions respectively, where $H$ is a positive constant, $\Omega$ is a bounded smooth convex domain in $\mathbb{R}^{n}(n\geq2)$.

Critical points of solutions to elliptic equations is a significant research topic. In some cases, properties of critical points of solutions are themselves the main concern. In other cases, theory of critical points provide an important tool in the study of properties of solutions. There are many known results about the study of critical points. In 1971 Makar-Limanov \cite{Makar-Limanov} investigated the Poisson equation with constant inhomogeneous term in a planar convex domain, and proved that $u$ has one unique critical point. In 1992 Alessandrini and Magnanini \cite{AlessandriniMagnanini1} considered the geometric structure of the critical set of a solution to semilinear elliptic equation in a nonconvex domain $\Omega,$ whose boundary is composed of finite simple closed curves. They deduced that the critical set is made up of finitely many isolated critical points. In 2012 Arango and G\'{o}mez \cite{ArangoGomez2} considered the critical points of the solutions for quasilinear elliptic equations with Dirichlet boundary condition in strictly convex domains and nonconvex domains respectively. If the domain is strictly convex and $u$ is a negative solution, they proved that the critical set has exactly one non-degenerate critical point. On the other hand, they obtained the similar results of the semilinear case for a planar annular domain, whose boundary has nonzero curvature. In 2017 Deng, Liu and Tian \cite{Deng2} investigate the geometric structure of interior critical points of solutions $u$ to a quasilinear elliptic equation with nonhomogeneous Dirichlet boundary conditions in a simply connected or multiply connected domain $\Omega$ in $\mathbb{R}^2$. They develop a new method to prove $\Sigma_{i = 1}^k {{m_i}}+1=N$ or $\Sigma_{i = 1}^k {{m_i}}=N,$ where $m_1,\cdots,m_k$ are the respective multiplicities of interior critical points $x_1,\cdots,x_k$ of $u$ and $N$ is the number of global maximum points of $u$ on $\partial\Omega$. All the above results involved the critical points of solutions to elliptic equations with Dirichlet boundary condition in the planar domains. In addition, a number of other authors have investigated this problem and some other related problems (see \cite{AlessandriniMagnanini2,Alessandrini1,Alessandrini2,Alessandrini3,Cecchini,Enciso1,Enciso2,Kawohl,Kraus,Lopez,Zhu}). However, the critical set $K$ has not been fully considered for some general cases, especially for higher dimensional spaces, Neumann and Robin boundary value problems.

For higher dimensional cases, there has few results about the critical points of solutions to elliptic equations. In 1998 Cabr\'{e} and Chanillo \cite{CabreChanillo}, under the assumption of the existence of a semi-stable solution, showed that the solution of Poisson equation $-\triangle u=f(u)$ has exactly one non-degenerate critical point in a smooth bounded convex domain of $\mathbb{R}^{n}(n\geq2)$. Recently, the authors \cite{Deng1} investigated the geometric structure of critical points of solutions to mean curvature equations with Dirichlet boundary condition over a smooth bounded domain $\Omega$ in $\mathbb{R}^n(n\geq 2).$

 Concerning the Neumann and Robin boundary value problems, the critical points of solutions to elliptic equations seems to be less considered. In 1990, Sakaguchi \cite{Sakaguchi} proved that the solutions of Poisson equation with Neumann and Robin boundary conditions respectively exist exactly one critical point. The goal of this paper is to obtain some results about the critical set of solutions to mean curvature equation with Neumann and Robin boundary conditions respectively in bounded smooth convex domains $\Omega$ of $\mathbb{R}^{n}(n\geq2)$.

Throughout this paper, we shall suppose that $\Omega$ is a bounded smooth convex domain. Our theorems concerns only qualitative properties of the critical points of solutions to the prescribed constant mean curvature equation, i.e., the uniqueness and non-degeneracy. Hence we only need the hypothesis of the existence of solutions. The existence of solutions can be seen in \cite{Gilbarg,MaXu,PucciSerrin} etc. Our main results are as follows:
\begin{Theorem}\label{th1.1} 
Let $\Omega$ be a bounded smooth convex domain in $\mathbb{R}^2.$ Suppose that $H$ is a positive constant and that $u$ is a solution of the following boundary value problem
\begin{equation}\label{1.2}
\left\{
\begin{array}{l}
{\rm div}(\frac{\nabla u}{\sqrt{1+|\nabla u|^2}})=H~~~\mbox{in}~\Omega,\\
\frac{\partial u}{\partial \overrightarrow{n}}=c ~~~\mbox{on}~ \partial\Omega,
\end{array}
\right.
\end{equation}
where $\overrightarrow{n}$ is the unit outward normal vector of $\partial\Omega$ and $c$ is a positive constant. Then $u$ has exactly one critical point $p$ in $\Omega$ and $p$ is a non-degenerate interior minimal point of $u$.
\end{Theorem}

\begin{Theorem}\label{th1.2} 
Let $\Omega$ be a bounded smooth convex domain in $\mathbb{R}^2.$ Suppose that $H$ is a positive constant. Let $u$ be a solution of the following boundary value problem
\begin{equation}\label{1.3}
\left\{
\begin{array}{l}
{\rm div}(\frac{\nabla u}{\sqrt{1+|\nabla u|^2}})=H~~~\mbox{in}~\Omega,\\
\frac{\partial u}{\partial \overrightarrow{n}}+\alpha u=0 ~~~\mbox{on}~ \partial\Omega,
\end{array}
\right.
\end{equation}
where $\alpha$ is a positive constant. Then $u$ has exactly one critical point $p$ in $\Omega$ and $p$ is a non-degenerate interior minimal point of $u$.
\end{Theorem}

In addition, we will give partial results about the critical points of solutions in higher dimensional spaces.

\begin{Theorem}\label{th1.3} 
 Let $\Omega$ be a bounded smooth convex domain of rotational symmetry with respect to $x_n$ axis in $\mathbb{R}^n(n\geq3).$ Suppose that $u=u(|x'|,x_n)$ is an axisymmetric solution of equation (\ref{1.2}) or (\ref{1.3}), where $|x'|=\sqrt{x_1^2+x_2^2+\cdots+x_{n-1}^2}$. Then $u$ has exactly one critical point $p$ in $\Omega$ and $p$ is a non-degenerate interior minimal point of $u$.
\end{Theorem}

The rest of this paper is organized as follows. In Section 2, we introduce the local Chen \& Huang's comparison technique. In Section 3, firstly, we prove that every interior critical point of $u$ is non-degenerate. Then, by the strong maximum principle and Hopf lemma, we show that $u$ does not have maximum points in $\Omega$ and that $u$ cannot have minimal points on $\partial\Omega$. Moreover, we prove the sufficient and necessary condition for existence of the saddle points of $u$. In Section 4, firstly, we show the uniqueness of the interior minimal points of $u$ by continuity argument. Then, by the sufficient and necessary condition for existence of the saddle points and the non-degeneracy of interior critical points in Section 3, we prove the uniqueness of the interior critical points of $u$. In Section 5, our main idea is to study the projection of higher dimensional spaces onto two dimensional plane. So we need to consider the domains $\Omega$ of revolution formed by taking a strictly convex planar domain about one axis. We deduce that $u$ has exactly one critical point $p$ in a bounded smooth strictly convex domain of $\mathbb{R}^{n}(n\geq3)$ and $p$ is a non-degenerate interior minimal point of $u$.

\section{Local Chen and Huang's comparison technique}
~~~~In order to obtain the non-degeneracy and uniqueness of the critical points of $u$ in planar domains. In this section, we will recall the key local Chen \& Huang's comparison technique in \cite{ChenHuang}. For the sake of completeness, in our setting, we will give a complete proof of Lemma 1 in \cite{ChenHuang}.

\begin{Lemma}\label{le2.1} 
 Suppose that $u, v$ satisfy the same constant mean curvature equation (\ref{1.1}). Without loss of generality, we suppose that $u,v$ have a second order contact at $Z_0=({x_{1}}_0,{x_{2}}_0,u({x_{1}}_0,{x_{2}}_0))$ with $({x_{1}}_0,{x_{2}}_0)=(0,0).$  Then by changing coordinate $(x_1,x_2)$ into $(\xi, \eta)$ linearly, the difference $u-v$ around $(\xi, \eta)=(0,0)$ is given by

\begin{equation}\label{2.1}\begin{array}{l}
u-v={\rm Re}(\rho\cdot(\xi+\eta i)^k)+o((\xi^2+\eta^2)^{\frac{k}{2}}),
\end{array}\end{equation}
where $k\geq 3,$ $\rho$ is a complex number and $\xi+\eta i$ is the complex coordinate.

\end{Lemma}
\begin{proof}[Proof] Since  $u, v$ satisfy the same constant mean curvature equation. Then we have
\begin{equation}\label{2.2}
\begin{split}
0&=(1+u_{x_1}^2+u_{x_2}^2)(u_{x_1x_1}+u_{x_2x_2})-(u_{x_1}^2u_{x_1x_1}+u_{x_2}^2u_{x_2x_2}+2u_{x_1}u_{x_2}u_{x_1x_2})-H(1+|\nabla u|^2)^{\frac{3}{2}}\\
&=(1+u_{x_2}^2)u_{x_1x_1}+(1+u_{x_1}^2)u_{x_2x_2}-2u_{x_1}u_{x_2}u_{x_1x_2}-H(1+|\nabla u|^2)^{\frac{3}{2}},
\end{split}
\end{equation}
and
\begin{equation}\label{2.3}
\begin{array}{l}
0=(1+v_{x_2}^2)v_{x_1x_1}+(1+v_{x_1}^2)v_{x_2x_2}-2v_{x_1}v_{x_2}v_{x_1x_2}-H(1+|\nabla v|^2)^{\frac{3}{2}}.
\end{array}
\end{equation}
Now we define $p(t),q(t),m(t),r(t),s(t)$ for $0\leq t\leq 1$ by
\begin{equation*}
\begin{array}{l}
p(t)=(1-t)v_{x_1x_1}+tu_{x_1x_1},~~q(t)=(1-t)v_{x_1x_2}+tu_{x_1x_2},~~m(t)=(1-t)v_{x_2x_2}+tu_{x_2x_2},\\
r(t)=(1-t)v_{x_1}+tu_{x_1},~~s(t)=(1-t)v_{x_2}+tu_{x_2},
\end{array}
\end{equation*}
and consider the following function
\begin{equation*}
\begin{array}{l}
W=W(p(t),q(t),m(t),r(t),s(t))=(1+s^2)p+(1+r^2)m-2rsq-H(1+r^2+s^2)^{\frac{3}{2}}.
\end{array}
\end{equation*}
 Let $w=u-v,$ therefore by (\ref{2.2}) minus (\ref{2.3}) we have
\begin{equation*}
\begin{split}
0&=W(p(1),q(1),m(1),r(1),s(1))-W(p(0),q(0),m(0),r(0),s(0))=\int_0^1 W_t dt\\
&=a_{11}w_{x_1x_1}+2a_{12}w_{x_1x_2}+a_{22}w_{x_2x_2}+b_{1}w_{x_1}+b_{2}w_{x_2},
\end{split}
\end{equation*}
where
\begin{equation*}
\begin{array}{l}
a_{11}=\int_0^1(1+s^2)dt,~~a_{12}=-\int_0^1rsdt,~~a_{22}=\int_0^1(1+r^2)dt,\\
b_1=\int_0^1[2(rm-sq)-3H\sqrt{1+r^2+s^2}r]dt,\\
b_2=\int_0^1[2(sp-rq)-3H\sqrt{1+r^2+s^2}s]dt.\\
\end{array}
\end{equation*}
Then $w$ satisfies the following equation
\begin{equation*}
\begin{array}{l}
Lw:=a_{11}w_{x_1x_1}+2a_{12}w_{x_1x_2}+a_{22}w_{x_2x_2}+b_{1}w_{x_1}+b_{2}w_{x_2}=0,
\end{array}
\end{equation*}
where $a_{12}^2-a_{11}a_{22}<0$ ensures the ellipticity of the equation $Lw=0.$ Next, the rest of proof is same to that in \cite{ChenHuang}. We transform $(x_1,x_2)$ into $(\xi,\eta)$
such that $\xi(0,0)=0, \eta(0,0)=0$ and at $(0,0)$
\begin{equation}\label{2.4}
\begin{array}{l}
Lw=(\frac{\partial^2}{\partial \xi^2}+\frac{\partial^2}{\partial \eta^2}+b_1'\frac{\partial}{\partial \xi}+b_2'\frac{\partial}{\partial \eta})w=0.
\end{array}
\end{equation}
Since the coefficients of $Lw$ and $w$ itself are analytic in $(x_1,x_2)$ as well as in $(\xi,\eta),$ then we get the following Taylor expansion around $(\xi,\eta)=(0,0)$ of $Lw$:
\begin{eqnarray}\label{2.5}
Lw&=&\Big\{(1+\alpha_{11}\xi+\beta_{11}\eta+O(\xi^2+\eta^2))\frac{\partial^2}{\partial \xi^2}+(1+\alpha_{22}\xi+\beta_{22}\eta+O(\xi^2+\eta^2))\frac{\partial^2}{\partial \eta^2}\nonumber\\
  &&+2(\alpha_{12}\xi+\beta_{12}\eta+O(\xi^2+\eta^2))\frac{\partial^2}{\partial \xi\partial\eta}+(\tau_1+\delta_1\xi+\lambda_1\eta+O(\xi^2+\eta^2))\frac{\partial}{\partial \xi}\\
 &&+(\tau_2+\delta_2\xi+\lambda_2\eta+O(\xi^2+\eta^2))\frac{\partial}{\partial \eta}\Big\}w.\nonumber
\end{eqnarray}
By Theorem I in \cite{Bers}, we have
\begin{equation}\label{2.6}
\begin{array}{l}
w(\xi,\eta)=\sum_{j=0}^{\infty}P_{k+j}(\xi,\eta),
\end{array}
\end{equation}
where $P_k(\xi,\eta)$ is a non-zero homogeneous polynomial in $(\xi,\eta)$ of degree $k$. By the assumption of $u$ and $v$ have a second order contact at $(0,0),$ we have $k\geq 3.$ By (\ref{2.5}) and (\ref{2.6}), the equation (\ref{2.4}) yields
\begin{equation*}
\begin{array}{l}
0=(\frac{\partial^2}{\partial \xi^2}+\frac{\partial^2}{\partial \eta^2})P_k+\{\mbox{terms of order}\geq k-1\}.
\end{array}
\end{equation*}
By the uniqueness of the power expansion, we show that $P_k$ is a harmonic homogeneous polynomial. Then
\begin{equation}\label{2.7}
\begin{array}{l}
P_k(\xi,\eta)={\rm Re}(\rho\cdot(\xi+\eta i)^k),
\end{array}
\end{equation}
where $\rho$ is a complex number, then (\ref{2.6}) and (\ref{2.7}) imply (\ref{2.1}).
\end{proof}

\begin{Lemma}\label{le2.2}
(see\cite[Lemma 2]{ChenHuang}) Suppose that $u=u(x_1,x_2)$ is a non-constant solution of the following homogeneous quasilinear elliptic equation
\[Lu=a_{11}u_{x_1x_1}+2a_{12}u_{x_1x_2}+a_{22}u_{x_2x_2}+b_{1}u_{x_1}+b_{2}u_{x_2}=0~~\mbox{in}~~\Omega,\]
where the coefficients ${a_{ij}} ~\mbox{and}~{b_i}~(i,j=1,2)$ are analytic. Then every interior critical point of $u$ is an isolated critical point.
\end{Lemma}
\begin{Remark}\label{re2.3}
By the above two lemmas and the implicit function theorem, we can know that the nodal set $N\cap\Omega$ of $(u-v)$ consists of at least three smooth arcs intersecting at $(0,0)$ and dividing $\Omega$ into at least six sectors. Moreover, the nodal set $N$ of $(u-v)$ is globally a union of smooth arcs.
\end{Remark}

\section{The sufficient and necessary condition for existence of the saddle points}
~~~~In this section, firstly, we investigate the non-degeneracy of critical points in a planar bounded smooth convex domain $\Omega$ by using the local Chen \& Huang's comparison technique. Then we prove the sufficient and necessary condition for existence of the saddle points by using the geometric properties of approximate surfaces at the non-degenerate critical points.

\begin{Lemma}\label{le3.1} 
Suppose that $u$ is a solution to (\ref{1.3}). Then $u<0$ in $\overline{\Omega}$ and $\frac{\partial u}{\partial \overrightarrow{n}}>0$ on $\partial \Omega$.
\end{Lemma}
\begin{proof}[Proof] According to the assumption of
\[\left\{ \begin{array}{l}
  \mbox{div}(\frac{\nabla u}{\sqrt{1+|\nabla u|^2}})=\sum\limits_{i,j=1}^{n}
a_{ij}(\nabla u)\frac{\partial^2 u}{\partial x_i \partial x_j}=H >0~~ \mbox{in}~~\Omega, \\
 \frac{\partial u}{\partial \overrightarrow{n}}+\alpha u=0 ~~~\mbox{on}~ \partial\Omega,
 \end{array} \right.\]
where $a_{ij}=\frac{1}{\sqrt{1+|\nabla u|^2}}(\delta_{ij}-\frac{u_{x_i}u_{x_j}}{1+|\nabla u|^2})$. By the strong maximum principle, $u$ obtains its maximum on $\partial\Omega.$ In fact, by the positive definiteness of the matrix $A=(a_{ij})$, if $u$ obtains the maximum at $x_0\in \Omega,$ then $B=(D_{ij}u(x_0))$ is seminegative definite. Hence the matrix $AB$ is seminegative definite with a nonpositive trace, it implies that $\sum_{i,j=1}^{n}a_{ij}(\nabla u)\frac{\partial^2 u}{\partial x_i \partial x_j}\leq 0,$ which is a contradiction. Thus there exists a point $x_0\in \partial\Omega$ such that $u(x_0)=\max_{\overline{\Omega}}u.$ Suppose that $u(x_0)\geq 0.$ By Hopf lemma we have $\frac{\partial u(x_0)}{\partial \overrightarrow{n}}>0.$ This contradicts with the fact that $\frac{\partial u(x_0)}{\partial \overrightarrow{n}}+\alpha u(x_0)=0,$ thus $u<0~\mbox{in}~\overline{\Omega}.$ Therefore $\frac{\partial u}{\partial \overrightarrow{n}}=-\alpha u>0~\mbox{on}~\partial\Omega.$
\end{proof}

\begin{Lemma}\label{le3.2} 
Let $u$ be a solution to (\ref{1.2}), or (\ref{1.3}). Then $u$ has at least one critical point in $\Omega.$
\end{Lemma}
\begin{proof}[Proof] Since $\frac{\partial u}{\partial \overrightarrow{n}}>0~\mbox{on}~\partial\Omega,$ the Hopf lemma implies that $u$ cannot have minimal points on $\partial\Omega$. On the other hand, since $u$ is an analytic non-constant function, therefore $u$ must obtain its minimum at some $p\in\Omega$ with $\nabla u(p)=0.$ Then $u$ has at least one point $p$ with $\nabla u(p)=0.$
\end{proof}

\begin{Lemma}\label{le3.3} 
Let $u$ be a solution to (\ref{1.2}), or (\ref{1.3}). Then $u$ is a Morse function, i.e., the Gaussian curvature $K(p):=\det(D^2u(p))\neq 0$ for any critical point $p.$
\end{Lemma}
Morse and semi-Morse function are described in \cite{Bott,Cavicchioli}. In order to prove Lemma \ref{le3.3}, we need the following lemma.

\begin{Lemma}\label{le3.4} 
For constant $H$ is from (\ref{1.1}) and any constant $h$, there exists a number $T$ $(0<T<\infty)$ such that the following initial value problem
\begin{equation}\label{3.1}\begin{array}{l}
\left\{ \begin{array}{l}
X''(t)=H(1+|X'(t)|^2)^{\frac{3}{2}}, ~~-T<t<T,\\
X(0)=h, \\
X'(0)=0,
 \end{array} \right.
\end{array}\end{equation}
has a unique $C^{\infty}$-solution $X(t)$, which satisfies the following
\begin{equation}\label{3.2}
\begin{array}{l}
X(t)=X(-t),~~-T<t<T,
\end{array}
\end{equation}
\begin{equation}\label{3.3}
\begin{array}{l}
X(t)\geq  h,~~-T<t<T,
\end{array}
\end{equation}
\begin{equation}\label{3.4}
\begin{array}{l}
X'(t)\geq 0,~~0\leq t<T.
\end{array}
\end{equation}

\end{Lemma}
\begin{proof}[Proof] Since the solution of problem (\ref{3.1}) is $X(t)=h+\frac{1}{H}(1-\sqrt{1-H^2t^2})$ for $|t|<T=\frac{1}{H}.$ So the results naturally hold.
\end{proof}

\begin{proof}[Proof of Lemma \ref{le3.3}] We set up the usual contradiction argument. Suppose that $p\in \Omega$ is a point such that $\nabla u(p)=0$ and the Gaussian curvature $K(p)=0.$ Without loss of generality, by using a suitable parallel translation and a rotation of coordinates, we may suppose that
\begin{equation}\label{3.5}
\begin{array}{l}
p=0~~\mbox{and}~~[D_{ij}u(0)]=\mbox{diag}[H,0].
\end{array}
\end{equation}
By Lemma \ref{le3.4} for $h=u(0),$ we get a unique solution to (\ref{3.1}), denote by $v.$ Let $v(x)~(=v(x_1,x_2))=X(x_1),$ thus $v$ satisfies
\begin{equation}\label{3.6}
\left\{
\begin{array}{l}
\mbox{div}(\frac{\nabla v}{\sqrt{1+|\nabla v|^2}})=H, ~~\mbox{in}~(-T,T)\times \mathbb{R},\\
\left[D_{ij}v(0)\right]= \mbox{diag}[H,0],\\
v(0)=u(0)=h~~\mbox{and}~~\nabla v(0)=\nabla u(0)=0.
\end{array}
\right.
\end{equation}
By (\ref{3.6}), we know that $(u-v)$ vanishes up to second order derivatives at 0. Moreover, we can know that $(u-v)$ is not identically zero. In fact, $v=v(x)=X(x_1).$ On the other hand, Lemma \ref{le3.1} shows that $\frac{\partial u}{\partial \overrightarrow{n}}>0,$ we can suppose that unit outward normal vector $\overrightarrow{n}=(0,1)$, then $\frac{\partial u}{\partial \overrightarrow{n}}=\nabla u\cdot \overrightarrow{n}=u_{x_2}>0.$ So $(u-v)$ is not identically zero. The unique continuation theorem of solutions for elliptic equations shows that $(u-v)$ never vanishes up to infinite order at 0. Using Lemma \ref{le2.1}, we get
\begin{equation}\label{3.7}
\begin{array}{l}
(u-v)(x)=P_k(x)+o(|x|^k)~~\mbox{as}~|x|\rightarrow 0
\end{array}
\end{equation}
for some integer $k\geq 3,$ where $P_k(x)$ is a homogeneous polynomial of degree $k$ and $P_k(x)$ is not identically zero. In addition, Lemma \ref{le2.2} shows that every interior critical point of $(u-v)$ is isolated. Furthermore, Remark \ref{re2.3} shows that the nodal sets of $(u-v)$ consist of $k$ smooth arcs in some neighborhood $U$ of the origin, and that all smooth arcs intersect at $(0,0)$ and divide $U$ into $2k (k\geq 3)$ sectors.

Firstly, we investigate the case of Neumann boundary condition (\ref{1.2}). In order to prove the result, we should divide the proof into two cases, i.e., $T$ is large enough and not large enough respectively. Consider
\begin{equation}\label{3.8}
\begin{array}{l}
I_+=\Big\{x\in \Omega\cap \{(-T,T)\times\mathbb{R}\};u(x)-v(x)>0\Big\},
\end{array}
\end{equation}
and
\begin{equation}\label{3.9}
\begin{array}{l}
I_-=\Big\{x\in \Omega\cap \{(-T,T)\times\mathbb{R}\};u(x)-v(x)<0\Big\}.
\end{array}
\end{equation}
Therefore, it follows from Lemma \ref{le2.1} and Remark \ref{re2.3} that
\begin{equation}\label{3.10}
\begin{array}{l}
\mbox{Both}~I_+ ~\mbox{and}~ I_- ~\mbox{have\ at\ least\ three\ components\ and\ each}\\
\mbox{of\ them\ meets\ the\ boundary} ~\partial( \Omega\cap \{(-T,T)\times\mathbb{R}\}).
\end{array}
\end{equation}

Case 1: If $T$ is large enough, i.e., $\Omega\cap \{(-T,T)\times\mathbb{R}\}=\Omega$. Since $\Omega$ is convex, by Lemma \ref{le3.1} and Lemma \ref{le3.4} we have that $\frac{\partial u}{\partial \overrightarrow{n}}>0$ and $v'(x_1)\geq 0, v''(x_1)\geq 0 ~\mbox{for}~ 0\leq x_1<T$, then we know that $\frac{\partial (u-v)}{\partial \overrightarrow{n}}$ preserves sign in some fixed arc and $\frac{\partial (u-v)}{\partial \overrightarrow{n}}$ changes sign alternatively in two adjacent arcs. The sign distribution for directional derivative of $(u-v)$ on $\partial\Omega$ is shown in Fig. 1.
\begin{center}
  \includegraphics[width=5cm,height=3.7cm]{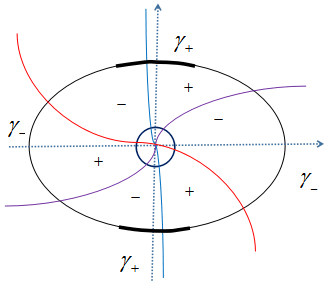}\\
  \scriptsize {\bf Fig. 1.}~~The sign distribution for directional derivative of $(u-v)$ on $\partial\Omega$.
\end{center}
Now we put
\begin{equation}\label{3.11}
\begin{array}{l}
\gamma_+=\Big\{x\in \partial\Omega;\frac{\partial (u-v)}{\partial \overrightarrow{n}}(x)>0\Big\},
\end{array}
\end{equation}
and
\begin{equation}\label{3.12}
\begin{array}{l}
\gamma_{-}=\Big\{x\in \partial\Omega;\frac{\partial (u-v)}{\partial \overrightarrow{n}}(x)<0\Big\}.
\end{array}
\end{equation}
Therefore, it never occurs that a component of $(I_-\cap\Omega)$ meets $\partial\Omega$ exclusively in $\gamma_+.$  Suppose by contradiction that $\gamma_1$ is a component of ($I_-\cap\Omega$) which meets $\partial\Omega$ exclusively in $\gamma_+.$ By Lemma \ref{le2.1}, we see that $(u-v)$ satisfies (\ref{3.7}), then the strong maximum principle implies that a negative minimum of $(u-v)$ in $\overline{\gamma_1}$ is attained at $p\in\gamma_+$ and $\frac{\partial (u-v)}{\partial \overrightarrow{n}}(p)\leq0$. This contradicts with the definition of $\gamma_+.$ By the same way, we know that it never occurs that a component of $(I_+\cap\Omega)$ meets $\partial\Omega$ exclusively in $\gamma_-.$ But these facts contradict with (\ref{3.10})

Case 2: If $T$ is not large enough, i.e., $\Omega\backslash \{\Omega\cap \{(-T,T)\times\mathbb{R}\}\neq \varnothing$. Choose a number $\widetilde{T}$ such that $\widetilde{T}<T,$ which is sufficiently near to $T.$ Set $\widetilde{\Omega}=(-\widetilde{T},\widetilde{T})\times\mathbb{R}.$ We only should replace $\Omega$ by $\Omega\cap\widetilde{\Omega}$ and we can use the same method of case 1.

Secondly, we consider the case of Robin boundary condition (\ref{1.3}). We can use the same method in the situation of Neumann boundary condition. Indeed, we select $\widetilde{T}$ with $\{v(x)=0\}=\{x_1=\pm \widetilde{T}\}$. Thus  $\partial(\Omega\cap\widetilde{\Omega})$ consists of
$$\mbox{at\ most\ two\ components\ of}~ \Big\{\frac{\partial (u-v)}{\partial \overrightarrow{n}}+\alpha (u-v)>0\Big\}$$
and
$$\mbox{at\ most\ two\ components\ of}~ \Big\{\frac{\partial (u-v)}{\partial \overrightarrow{n}}+\alpha (u-v)<0\Big\}.$$
\end{proof}

\begin{Lemma}\label{le3.5} 
Let $u$ be a solution to (\ref{1.2}), or (\ref{1.3}). Then $u$ does not have maximum points in $\Omega.$
\end{Lemma}
\begin{proof}[Proof] Since $\mbox {div}(\frac{\nabla u}{\sqrt{1+|\nabla u|^2}})=H>0$, the strong maximum principle implies that $u$ cannot obtain maximum in $\Omega.$  In fact,  let operator $L$ be the mean curvature operator. Suppose that $u$ has maximum points in $\Omega$ and that $x_0$ is a maximum point, then $D_1u(x_0)=D_2u(x_0)=0, D_{11}u(x_0)\leq 0$, $D_{22}u(x_0)\leq 0$ and $Lu(x_0)=D_{11}u(x_0) + D_{22}u(x_0)\leq 0$. However, this is contradict with $Lu>0$ in $\Omega.$ This completes the proof of Lemma \ref{le3.5}.
\end{proof}

Next, we show the sufficient and necessary condition for existence of the saddle points.
\begin{Lemma}\label{le3.6} 
Let $u$ be a solution to (\ref{1.2}), or (\ref{1.3}). Then $u$ has at least two minimal points in $\Omega$, if and only if there exists a point $p$ such that $\nabla u(p)=0$ and $K(p)<0.$
\end{Lemma}
\begin{proof}[Proof] (i) Firstly, we prove the section of ``if". Suppose that $p$ is a point such that $\nabla u(p)=0~\mbox{and}~K(p)<0.$ Therefore there exists an open neighborhood $U$ of $p$ in which the nodal sets of $(u-u(p))$ consist of at least two smooth arcs intersecting at $p$ and divides $U$ into at least four sectors.
Next, we consider the following super-level set
\begin{equation*}
\begin{array}{l}
U_+:=\{x\in \Omega|u(x)>u(p)\}.
\end{array}
\end{equation*}

Lemma \ref{le3.5} implies that each component of $U_+$ has to meet the $\partial\Omega.$ Therefore we know that the following sub-level set
\begin{equation*}
\begin{array}{l}
U_-:=\{x\in \Omega|u(x)<u(p)\}
\end{array}
\end{equation*}
has at least two components. By $\frac{\partial u}{\partial \overrightarrow{n}}>0,$ then $u$ has at least two minimal points in $\Omega$.

(ii) Secondly, we prove the section of ``only if". Since $\frac{\partial u}{\partial \overrightarrow{n}}>0~\mbox{on}~\partial\Omega$ and $\Omega$ is convex, so we can extend the solution $u$ to $\mathbb{R}^2$, denoting by
\begin{equation}\label{3.13}\begin{array}{l}
u(x)=u(y)+\mbox{dist}(x,y)\cdot\frac{\partial u}{\partial \overrightarrow{n}}(y) ,
\end{array}\end{equation}
where $y$ is the unique point on $\partial\Omega$ such that $\mbox{dist}(x,y)=\mbox{dist}(x,\Omega).$ Hence we know that $u\in C^1(\mathbb{R}^2)~\mbox{and}~\nabla u\neq 0~\mbox{in}~\mathbb{R}^2\setminus\Omega.$ Next, we consider the sub-level set $N_z=\{x\in \mathbb{R}^2|u(x)<z\}.$ Therefore we have that
\begin{equation}\label{3.14}\begin{array}{l}
\partial N_z~\mbox{has\ only\ one\ curve\ for\ sufficiently\ large}~z.
\end{array}\end{equation}

Now, we suppose by contradiction that $u$ has at least two minimal points and there does not exist point $p$ such that $\nabla u(p)=0~\mbox{and}~K(p)<0.$
By Lemma \ref{le3.3}, Lemma \ref{le3.5} and (\ref{3.13}), we know that each critical point of $u$ is a minimal point and $\nabla u\neq 0~\mbox{in}~\mathbb{R}^2\backslash \Omega$ respectively. On the other hand, Lemma \ref{le2.2} shows that every critical point is isolated and the number of critical points is finite. Then we suppose that there exists a finite sequence of minimal points of $u$, denotes by $\{p_1,p_2,\ldots,p_k\}$ such that
\begin{equation}\label{3.15}\begin{array}{l}
\nabla u(x)\neq 0~\mbox{for\ all}~x\in \mathbb{R}^2\backslash\{p_1,p_2,\ldots,p_k\}.
\end{array}\end{equation}
Let $z_0=\max\{u(p_i)|1\leq i\leq k\}.$ Therefore we know that the boundary $\partial N_z$ of the sub-level set $N_z$ is $C^1$ curve for $z>z_0$ and $\{\partial N_z\}$ is diffeomorphic to each other. According to the assumption, since $K(p_i)>0,$ then the approximate surface is an elliptic paraboloid in a neighborhood of critical point $p_i$ (If $K(p)<0,$ then the approximate surface is a hyperbolic paraboloid in a neighborhood of critical point $p$). The elliptic paraboloid and hyperbolic paraboloid as shown in Fig. 2 and Fig. 3, respectively.
\begin{center}
  \includegraphics[width=6.6cm,height=4.5cm]{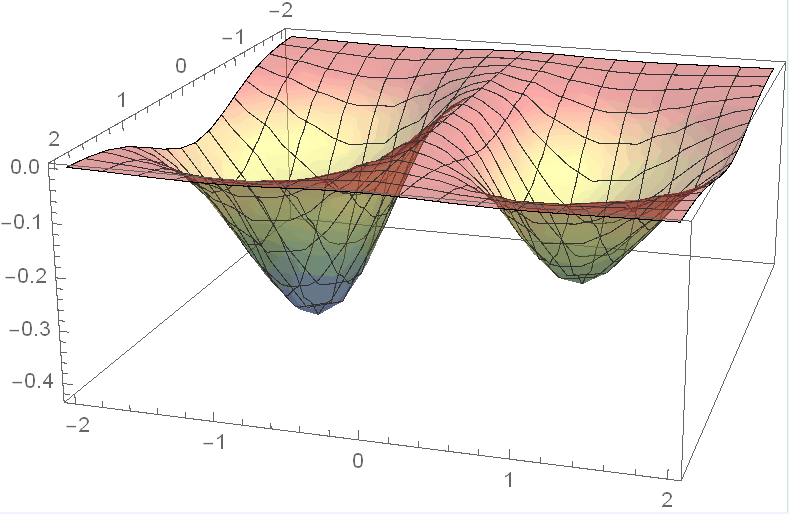}\includegraphics[width=6.6cm,height=4.5cm]{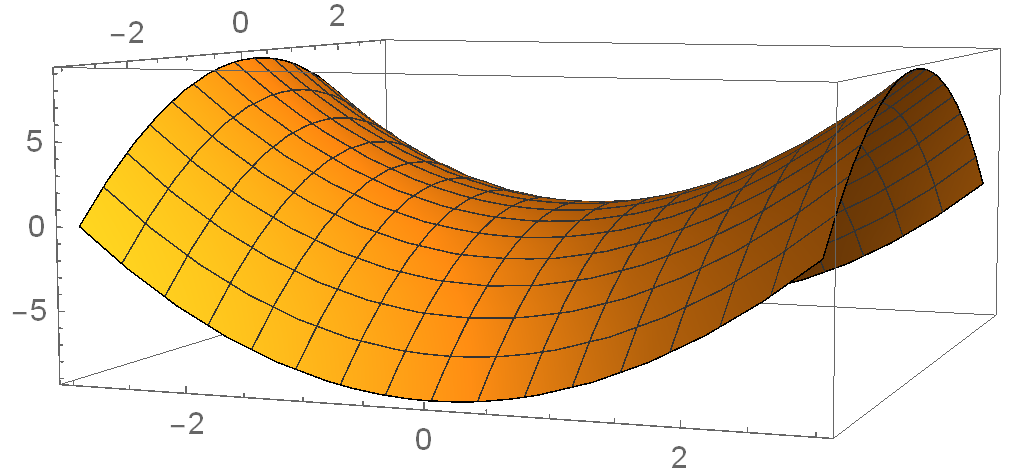}\\
  \scriptsize {\bf Fig. 2.}~~The elliptic paraboloid. ~~~~~~~~~~~~~~~~~~~~~~~~~~~~{\bf Fig. 3.}~~The hyperbolic paraboloid.
\end{center}
If $z$ is near to $z_0,$ then  $\{\partial N_z\}$ has at least two curves. This contradicts with the fact (\ref{3.14}). This completes the proof of Lemma \ref{le3.6} .
\end{proof}

\section{The proof of Theorem \ref{th1.1} and Theorem \ref{th1.2}}
~~~~In this section, firstly, we show the uniqueness of the interior minimal points of $u$ in $\Omega$ by continuity argument. Then, by the sufficient and necessary condition for existence of the saddle points and the non-degeneracy of interior critical points in Section 3, we prove the uniqueness of the critical points.

For $t$ ($t\in [0,1]$), we will prove the uniqueness of the interior minimal points of the solutions $v_t$ to the following problems:
\begin{equation}\label{4.1}\begin{array}{l}
\left\{ \begin{array}{l}
 \mbox{div}(\frac{\nabla v }{\sqrt{1 +t^2|\nabla v|^2}}) = H, ~\mbox{in}~\Omega,\\
 \frac{\partial v}{\partial \overrightarrow n} = c,~\mbox{on}~\partial\Omega,
 \end{array} \right.~~\mbox{or}~~
 \left\{ \begin{array}{l}
 \mbox{div}(\frac{\nabla v}{\sqrt {1+t^2|\nabla v|^2}}) = H,~\mbox{in}~\Omega, \\
 \frac{\partial v}{\partial \overrightarrow n} + \alpha v = 0,~\mbox{on}~\partial\Omega.
 \end{array} \right.
\end{array}\end{equation}
Let $u_t=tv_t$ for $t>0,$ then $u_t$ respectively satisfy
\begin{equation}\label{4.2}\begin{array}{l}
\left\{ \begin{array}{l}
 \mbox{div}(\frac{\nabla u}{\sqrt{1+|\nabla u|^2}}) =tH, ~\mbox{in}~\Omega,\\
 \frac{\partial u}{\partial \overrightarrow n}=c,~\mbox{on}~\partial\Omega,
 \end{array} \right.~~\mbox{or}~~
 \left\{ \begin{array}{l}
 \mbox{div}(\frac{\nabla u}{\sqrt {1+|\nabla u|^2}})=tH,~\mbox{in}~\Omega, \\
 \frac{\partial u}{\partial \overrightarrow n}+\alpha u=0,~\mbox{on}~\partial\Omega.
 \end{array} \right.
\end{array}\end{equation}

By the analysis of Lemma \ref{le3.2}, we know that the solution $v_t$ of (\ref{4.1}) has at least one minimal point in $\Omega.$ According to Lemma \ref{le3.3} and Lemma \ref{le3.6}, we can easily get the following lemmas for $v_t$.
\begin{Lemma}\label{le4.1} 
For any $t\in(0,1].$ Let $v_t$ be a solution to (\ref{4.1}). Then $v_t$ is a Morse function, i.e., the Gaussian curvature $K_t(p):=\det(D^2v_t(p))\neq 0~\mbox{for\ any\ critical\ point}~p.$
\end{Lemma}
\begin{Lemma}\label{le4.2} 
For any $t\in[0,1].$ Then $v_t$ has at least two minimal points, if and only if there exists a point $p$ such that $\nabla v_t(p)=0$ and $K_t(p)<0.$
\end{Lemma}

Next, we will prove that the Gaussian curvature of $v_t$ is positive at any critical point for the case of $t=0.$
\begin{Lemma}\label{le4.3} 
Let $v_t$ be the solution of (\ref{4.1}). Suppose that $\nabla v_0(p)=0$ for some point $p\in \Omega.$ Then the Gaussian curvature $K_0(p)>0$ for any critical point $p$.
\end{Lemma}
\begin{proof}[Proof]
We set up the usual contradiction argument. Suppose that there exists interior critical point $p$ such that $K_0(p)\leq0$. Without loss of generality, by using a suitable parallel translation and a rotation of coordinates, we may suppose that
\[p=(0,0), [D_{ij}v_0(p)]=diag[\lambda_1,\lambda_2],~\mbox{where}~\lambda_1+\lambda_2=H>0,\lambda_1>0,\lambda_2\leq0.\]
By Lemma \ref{le2.1}, then the difference $v_0-q$ around $(x_1,x_2)=(0,0)$ is given by
 $$v_0(x_1,x_2)-q(x_1,x_2)=P_k(x_1,x_2)+o((x_1^2+x_2^2)^{\frac{k}{2}}),$$
 where $P_k(x_1,x_2)$ is a homogeneous harmonic polynomial of degree $k$ in $\Omega$, and $$q(x_1,x_2)=v_0(0,0)+\frac{1}{2}\lambda_1x_1^2+\frac{1}{2}\lambda_2x_2^2.$$

Firstly, we study the case of Neumann boundary condition in (\ref{4.1}). Next, we consider
 \begin{equation*}
\begin{array}{l}
\widehat{I}_+=\big\{x\in \Omega;v_0(x)-q(x)>0\big\},
\end{array}
\end{equation*}
and
\begin{equation*}
\begin{array}{l}
\widehat{I}_-=\big\{x\in \Omega;v_0(x)-q(x)<0\big\}.
\end{array}
\end{equation*}
Since $v_0(x_1,x_2)-q(x_1,x_2)$ vanishes up to second order derivatives at $(0,0)$ and $P_k(x_1,x_2)$ is real analytic. Then it follows from Lemma \ref{le2.1} and Remark \ref{re2.3} that $k\geq 3$ and
\begin{equation}\label{4.3}
\begin{array}{l}
\mbox{Both}~\widehat{I}_+ ~\mbox{and}~ \widehat{I}_- ~\mbox{have\ at\ least\ three\ components}\\
\mbox{ and\ each\ of\ them\ meets\ the\ boundary} ~\partial\Omega.
\end{array}
\end{equation}
Now we set
\begin{equation}\label{4.4}
\begin{array}{l}
\widehat{\gamma}_+=\Big\{x\in \partial\Omega;\frac{\partial (v_0-q)}{\partial \overrightarrow{n}}(x)>0\Big\},
\end{array}
\end{equation}
and
\begin{equation}\label{4.5}
\begin{array}{l}
\widehat{\gamma}_{-}=\Big\{x\in \partial\Omega;\frac{\partial (v_0-q)}{\partial \overrightarrow{n}}(x)<0\Big\}.
\end{array}
\end{equation}
Since $\Omega$ is convex and $q(x_1,x_2)=v_0(0,0)+\frac{1}{2}\lambda_1x_1^2+\frac{1}{2}\lambda_2x_2^2$ with $\lambda_1+\lambda_2=H>0,\lambda_1>0,\lambda_2\leq0,$ then we know that $\widehat{\gamma}_+$ and $\widehat{\gamma}_{-}$ has at most two components on $\partial\Omega.$ The rest of the proof is same as the proof of Lemma \ref{le3.3}. This contradicts with (\ref{4.3}).

Secondly, we consider the case of Robin boundary condition in (\ref{4.1}). We can use the same method in the case of Neumann boundary condition. This completes the proof.
\end{proof}

Next we will show the uniqueness of the interior minimal points of $u$ in $\Omega$ by using the continuity argument.

\begin{Lemma}\label{le4.4} 
For any $t\in[0,1].$ Then $v_t$ has exactly one minimal point in $\Omega.$
\end{Lemma}
\begin{proof}[Proof] We set $M=[0,1]$ and divide $M$ into two sets $M_1$ and $M_2$ as follows:
\begin{equation}\label{4.6}
\begin{array}{l}
M_1=\{t\in M; v_t~\mbox{has\ only\ one\ minimal\ point\ in}~\Omega\},
\end{array}
\end{equation}
and
\begin{equation}\label{4.7}
\begin{array}{l}
M_2=\{t\in M; v_t~\mbox{has\ more\ than\ two\ minimal\ points\ in}~\Omega\}.
\end{array}
\end{equation}
Then $M=M_1+M_2$ and $M_1\cap M_2=\varnothing.$ Lemma \ref{le4.2} and Lemma \ref{le4.3} imply that $0\in M_1$, i.e., $M_1\neq \varnothing.$

Now we show that $M_2$ is open in $M.$ That is, for any $t_\star\in M_2,$ there exists a constant $\varepsilon>0$ with $(t_\star-\varepsilon,t_\star+\varepsilon)\subset M_2.$ In fact, the follows from Lemma \ref{le4.1} and Lemma \ref{le4.2} and inverse function theorem that $v_t$ has as many critical points as $v_{t_\star}$ when $t$ is near $t_\star.$ Suppose by contradiction that there exists a sequence $\{t_k\}\in M_1$ such that $\{v_{t_k}\}$ has only one minimal point and $t_k\in (t_\star-\frac{1}{k},t_\star+\frac{1}{k})$ for some positive $t_\star\in M_2$. Then it follows from Lemma \ref{le4.1} and Lemma \ref{le4.2} that $v_{t_k}$ does not has the saddle points, i.e., $v_{t_k}$ has exactly one critical point. By Lemma \ref{le4.1} and continuity, we may take a subsequence $\{v_{t_{k_j}}\}$ of $\{v_{t_k}\}$ such that
\begin{equation}\label{4.8}
\begin{array}{l}
p_{k_j}\rightarrow p,~~\nabla v_{t_{k_j}}(p_{k_j})=0,~~K_{t_{k_j}}(p_{k_j})>0,~~\nabla v_{t_\star}(p)=0,~~K_{t_\star}(p)>0.
\end{array}
\end{equation}
Since $t_\star\in M_2,$ then there exists another point $q\in U(p)\subset \Omega$ and a sequence of point $\{q_{k_j}\}$ such that
\[q_{k_j}\rightarrow q,~~\nabla v_{t_{k_j}}(q_{k_j})\rightarrow \nabla v_{t_\star}(q)=0.\]
According to the uniqueness of the critical point of $v_{t_k},$ we can take a subsequence $\{v_{t_{k_j}}\}$ of $\{v_{t_k}\}$ such that $v_{t_{k_j}}$ are all monotone in the line $\gamma(p_{k_j},q_{k_j})$. Therefore there exists a sequence of points $\{z_{k_j};z_{k_j}\in \gamma(p_{k_j},q_{k_j})\}$ which satisfy
\begin{equation}\label{4.9}
\begin{array}{l}
|\nabla v_{t_{k_j}}(z_{k_j})|\leq |\nabla v_{t_{k_j}}(q_{k_j})|\rightarrow 0,~~|K_{t_{k_j}}(z_{k_j})|=\frac{|\nabla v_{t_{k_j}}(q_{k_j})|}{|p_{k_j}-q_{k_j}|}\rightarrow 0.
\end{array}
\end{equation}
By (\ref{4.9}) and continuity, then there should be a point $z\in \gamma (p,q)$ such that
\[\nabla v_{t_\star}(z)=0,~~K_{t_\star}(z)=0,\]
this contradicts with Lemma \ref{le4.1}, then we complete the proof which $M_2$ is open set in $M.$

On the other hand, we show that $M_2$ is closed in $M.$ In fact, let $\{t_i\}$ be a sequence of points in $M_2$ such that $t_i\rightarrow t_0~\mbox{as}~i\rightarrow \infty.$ Then Lemma \ref{le4.2} and the continuity argument imply that there exists a subsequence $\{t_j\}$ of $\{t_i\}$, a sequence $\{p_j\}$ and a point $p\in\overline{\Omega}$ such that
\begin{equation}\label{4.10}
\begin{array}{l}
p_j\rightarrow p~\mbox{as}~j\rightarrow\infty,~~\nabla v_{t_j}(p_j)=0,~\mbox{and}~K_{t_j}(p_j)<0.
\end{array}
\end{equation}
By (\ref{4.10}) and continuity, we have
\begin{equation}\label{4.11}
\begin{array}{l}
\nabla v_{t_0}(p)=0,~\mbox{and}~K_{t_0}(p)\leq 0.
\end{array}
\end{equation}
Since $\nabla v_{t_0}\neq 0~\mbox{on}~\partial\Omega,$ then we have $p\in \Omega.$ Hence it follows from Lemma \ref{le4.1}, Lemma \ref{le4.2}, Lemma \ref{le4.3} and  (\ref{4.11}) that $t_0\in M_2.$ This shows that $M_2$ is closed in $M.$ Then $M_2$ must be $M$ or $\varnothing.$ Since $M_1\neq\varnothing,$ so $M_2=\varnothing~\mbox{and}~M_1=M.$ This completes the proof.
\end{proof}

\begin{proof}[Proof of Theorem \ref{th1.1} and Theorem \ref{th1.2}] By Lemma \ref{le3.3} and Lemma \ref{le3.5}, we know that the Gaussian curvature $K(p)\neq0$ for any critical point $p$ and solution $u$ does not have maximum points in $\Omega$. In addition, Lemma \ref{le3.6} shows that
 \begin{equation}\label{4.12}\begin{array}{l}
\mbox{if}~\exists~ p\in \Omega~ \mbox{such\ that} ~\nabla u(p)=0~ \mbox{and}~ K(p)<0\Leftrightarrow~\sharp\{\mbox{minimal\ points\ of}~u\}\geq 2.
\end{array}\end{equation}

 On the other hand, by Lemma \ref{le4.4}, we know that $u$ has exactly one minimal point in $\Omega.$ Therefore, $u$ does not have saddle points in $\Omega,$ this implies that $u$ has exactly one critical point $p$ in $\Omega$ and $p$ is a non-degenerate interior minimal point of $u$.
\end{proof}

\section{The proof of Theorem \ref{th1.3}} 

~~~~~In this section, we investigate the geometric structure of critical point set $K$ of solutions to prescribed constant mean curvature equation with Neumann and Robin boundary conditions respectively in higher dimensional spaces.

\begin{proof}[Proof Theorem \ref{th1.3}] We divide the proof into three steps.

 Step 1, we turn the mean curvature equation (\ref{1.1}) for $n$-dimension into the similar mean curvature equation for 2-dimension. Without loss generality, let $\Omega$ be a domain of revolution formed by taking a strictly convex planar domain in the $x_1,x_n$ plane with respect to the $x_n$ axis. In the sequel, $x=(x',x_n), x'=(x_1,\cdots,x_{n-1})$ and $r=\sqrt{x_1^2+\cdots+x_{n-1}^2}.$

By the assumptions, we have that the solution $u$ satisfies
\begin{equation}\label{5.1}\begin{array}{l}
u(x',x_n)=u(|x'|,x_n)\triangleq v(r,x_n)
\end{array}\end{equation}
and
\begin{equation}\label{5.2}\begin{array}{l}
\frac{\partial v}{\partial r}(r,x_n)>0$ $ $ for $ $ $ r\neq0.
\end{array}\end{equation}

From (\ref{5.2}), we can know that the critical points of $u$ lie on $x_n$ axis. Next, according to (\ref{5.1}), we have that
\begin{equation}\label{5.3}\begin{array}{l}
u_{x_n}(x',x_n)=v_{x_n}(r,x_n).
\end{array}\end{equation}
Moreover, we can deduce that $u_{x_n}$ satisfies the following equation
\[\sum\limits_{i,j=1}^{n}
a_{ij}(\nabla u)\frac{\partial^2 u_{x_n}}{\partial x_i \partial x_j}+\sum\limits_{i,j=1}^{n}
\frac{\partial a_{ij}(\nabla u) }{\partial x_n }\frac{\partial^2 u}{\partial x_i \partial x_j}=0~~~n\geq 3.\]
That is
\begin{equation}\label{5.4}\begin{array}{l}
\mathscr{L}u_{x_n}:=\sum\limits_{i,j=1}^{n}
a_{ij}(\nabla u)\frac{\partial^2 u_{x_n}}{\partial x_i \partial x_j}+\sum\limits_{i,j=1}^{n}\frac{\partial^2 u}{\partial x_i \partial x_j}
\frac{\partial a_{ij}(\nabla u)}{\partial x_n}=0,
\end{array}\end{equation}
where $a_{ij}(\nabla u)=\frac{1}{\sqrt{1+|\nabla u|^2}}(\delta_{ij}-\frac{u_{x_i}u_{x_j}}{1+|\nabla u|^2}),$ $\frac{\partial a_{ij}(\nabla u)}{\partial x_n }=\frac{1}{(1+|\nabla u|^2)^{3/2}}\big[(\frac{3u_{x_i}u_{x_j}}{1+|\nabla u|^2}-\delta_{ij})(\nabla u\cdot\nabla u_{x_n})-(u_{x_i}u_{x_nx_j}+u_{x_j}u_{x_nx_i})\big]$ is the first derivative term of $u_{x_n}$.

By the assumptions, the strict convexity of $\Omega$ and the Hopf lemma, we can know that $u_{x_n}$ vanishes precisely on the $(n-2)$ dimensional sphere given by
\[S=\{x_n=a\}\cap\partial\Omega,\]
for some $a\in \mathbb{R}.$ For convenience, we define the nodal set
\[\mathscr{N}=\{x\in \Omega|u_{x_n}(x)=0\}.\]
It is clear that all critical points of solution $u$ are contained in $\mathscr{N}$. Also from (\ref{5.3}), $\mathscr{N}$ is rotationally invariant about the $x_n$ axis.

Now we turn the mean curvature equation (\ref{1.1}) for $n$-dimension
\begin{equation*}
\begin{array}{l}
div(\frac{\nabla u}{\sqrt{1+|\nabla u|^2}})=H
\end{array}\end{equation*}
into the following similar mean curvature equation on 2-dimension
\[div(\frac{\nabla v}{\sqrt{1+|\nabla v|^2}})+\frac{1}{\sqrt{1+|\nabla v|^2}}\frac{n-2}{r}v_r=H,\]
that is
\begin{equation}\label{5.5}\begin{array}{l}
\sum\limits_{i,j=1}^{2}a_{ij}(\nabla v)v_{ij}+\frac{1}{\sqrt{1+|\nabla v|^2}}\frac{n-2}{r}v_r=H,
\end{array}\end{equation}
where $\nabla v=(\frac{\partial v}{\partial r},\frac{\partial v}{\partial x_n}),$  $a_{ij}=\frac{1}{\sqrt{1+|\nabla v|^2}}(\delta_{ij}-\frac{v_iv_j}{1+|\nabla v|^2})$ and $v_1=\frac{\partial v}{\partial r}, v_2=\frac{\partial v}{\partial x_n}.$

For any $\theta=(\theta_1,\theta_2)=(\cos \alpha,\sin \alpha)\in S^1,$ where $\alpha \in [0,\pi).$ We turn the quasilinear elliptic equation associated to $v$ into a linear elliptic equation associated to $w=v_{\theta}=\nabla v\cdot \theta.$ Firstly, we differentiate the equation (\ref{5.5}), then take inner product with $\theta.$ For convenience, we set $y=(y_1,y_2)=(x_n,r),$ hence we can get the following equation
\begin{equation}\label{5.6}\begin{array}{l}
L_{v}w+h_1(y)\frac{\partial w}{\partial y_1}+h_2(y)\frac{\partial w}{\partial y_2}+\frac{1}{(1+|\nabla v|^2)^{\frac{3}{2}}}\frac{n-2}{r}\Big[(1+v_{y_1}^2)\frac{\partial w}{\partial y_2}-v_{y_1}v_{y_2}\frac{\partial w}{\partial y_1}\Big] =\frac{1}{\sqrt{1+|\nabla v|^2}} \frac{n-2}{r^2}v_r \theta_2,
\end{array}\end{equation}
where
\[L_{v}w:=\sum\limits_{i,j=1}^{2}
a_{ij}(\nabla v)\frac{\partial^2 w}{\partial y_i \partial y_j}\]
and
\[h_{k}(y)=\sum\limits_{i, j=1}^{2} v_{y_i y_j}\frac{\partial a_{ij}}{\partial v_{y_k}},~~~k=1,2.\]
By (\ref{5.2}) and (\ref{5.6}), we deduce that
\begin{equation}\label{5.7}\begin{array}{l}
L_{v}w+h_1(y)\frac{\partial w}{\partial y_1}+h_2(y)\frac{\partial w}{\partial y_2}+\frac{1}{(1+|\nabla v|^2)^{\frac{3}{2}}}\frac{n-2}{r}\Big[(1+v_{y_1}^2)\frac{\partial w}{\partial y_2}-v_{y_1}v_{y_2}\frac{\partial w}{\partial y_1}\Big]\geq 0.
\end{array}\end{equation}
By (\ref{5.7}), so we can consider the result of projecting a graph onto  $x_1, x_n$ plane (see Fig. 4).
\begin{center}
  \includegraphics[width=8cm,height=3.5cm]{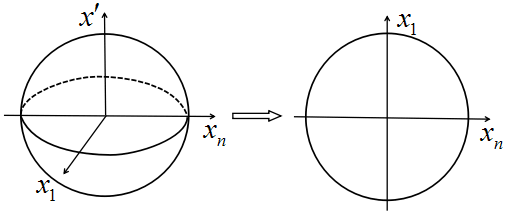}\\
  \scriptsize {\bf Fig. 4.}~~The graphic projection of higher dimensional space onto two dimensional plane.
\end{center}

 Step 2, we show the uniqueness of critical points. This subsection is based on the results of Gidas, Ni and Nirenberg \cite{Gidas} and Caffarelli and Friedman \cite{Caffarelli}, the ideas of Payne \cite{Payne} and Sperb \cite{Sperb}. To prove that whenever critical set has exactly one point, since all critical points of $u$ are contained in $\mathscr{N}\cap\{x_2=\cdots=x_{n-1}=0\}$ and lie on the $x_n$ axis. The nodal set $\mathscr{N}=\{x\in \Omega | u_{x_n}(x)=0\}$ is rotationally invariant about the $x_n$ axis, formed by a set $N_2$ contained in the $x_1,x_n$ 2-dimensional plane rotation about the $x_n$ axis, by (\ref{5.5}), where $N_2$ can be seen as the projection of $\mathscr{N}$ in the $x_1,x_n$ 2-dimensional plane and $\mathscr{N}$ cannot enclose any subdomain of $\Omega$ (By Lemma 2.1 in \cite{Deng1}, $N_2$ cannot enclose any planar subdomain of $\Omega\cap\{x_2=\cdots=x_{n-1}=0\},$ where locally $N_2$ looks like the nodal set of some homogeneous polynomial in $x_1,x_n.$). Because $N_2$ is symmetric with respect to the $x_n$ axis and intersects the $x_n$ axis at exactly one point, hence we prove the uniqueness of critical points.

Step 3, we show the non-degeneracy of critical point. How to show that critical point $p$ is non-degenerate, we restatement that $u$ is rotationally symmetric with respect to $x_n$ axis and critical point $p$ lies on this axis. From (\ref{5.1}) and (\ref{5.2}), we have that $\{u_{x_k}=0\}=\{x_k=0\}\cap \Omega$ for all $1\leq k\leq n-1.$ Hence $u_{x_ix_j}(p)=0$ for any index $1\leq i<j\leq n,$ that is, $D^2u(p)$ is diagonal. By (\ref{5.2}), we can know that $u_{x_k}>0$ in domain $\mathscr{D}_k=\{x_k>0\}\cap \Omega$ for $1\leq k\leq n-1.$ Furthermore, in domain $\mathscr{D}_k$, $u_{x_k}$ satisfies

\begin{equation}\label{5.8}\begin{array}{l}
\mathscr{L}u_{x_k}=\sum\limits_{i,j=1}^{n}
a_{ij}(\nabla u)\frac{\partial^2 u_{x_k}}{\partial x_i \partial x_j}+\sum\limits_{i,j=1}^{n}\frac{\partial^2 u}{\partial x_i \partial x_j}
\frac{\partial a_{ij}(\nabla u)}{\partial x_k}=0,
\end{array}\end{equation}
where $\frac{\partial a_{ij}(\nabla u)}{\partial x_k }=\frac{1}{(1+|\nabla u|^2)^{3/2}}\big[(\frac{3u_{x_i}u_{x_j}}{1+|\nabla u|^2}-\delta_{ij})(\nabla u\cdot\nabla u_{x_k})-(u_{x_i}u_{x_kx_j}+u_{x_j}u_{x_kx_i})\big]$ is the first derivative term of $u_{x_k}$. According to the Hopf lemma, we deduce that $u_{x_kx_k}(p)>0$ for all $1\leq k\leq n-1,$ where critical point $p\in \partial \mathscr{D}_k.$

Finally, we recall that the function $u_{x_n}$ satisfies (\ref{5.4}). By the definition of $\mathscr{N}$, $u_{x_n}>0$ to one side of $\mathscr{N}.$ By applying the Hopf lemma to $u_{x_n}$ at $p\in \mathscr{N},$ we have that $u_{x_nx_n}(p)>0.$ So we prove that the Hessian matrix $D^2u(x)$ of $u$ is diagonal and positive definite at critical point $p$, hence $p$ is the unique critical point and $p$ is a non-degenerate interior minimal point of $u$. This completes the proof of Theorem \ref{th1.3}.
\end{proof}

\end{document}